\documentclass[11pt,a4paper,reqno]{amsart}
\usepackage{color,latexsym,amsfonts,amssymb,bbm,comment}
\usepackage{hyperref}
\usepackage{amsmath,cite}
\usepackage{amsthm}
\usepackage{fullpage}
\usepackage{graphicx}
\usepackage{tikz}
\usepackage{caption,subcaption}
\usepackage{makecell}
\usepackage{boldline}
\setcellgapes{3pt}
\usepackage{multirow}
\usepackage{mathtools}

\newtheorem {thm}{Theorem}[section]
\newtheorem {prop}[thm]{Proposition} 

\newtheorem {lem}[thm]{Lemma}
\newtheorem {cor}[thm]{Corollary}
\newtheorem {defn}[thm]{Definition}

\def\N{{\Bbb N}}

\def\Z{{\Bbb Z}}
\def\R{{\Bbb R}}
\def\one{\mathbbmss{1}}

\def\P{{\Bbb P}}

\def\e{{\varepsilon}}

\def\D{\Delta}
\def\a{\alpha}
\def\sm{\setminus}

\def\EE{{\cal E}}

\def\d{\delta}

\def\e{\varepsilon}

\def\phi{\varphi}

\def\g{\gamma}

\def\l{\lambda}

\def\k{\kappa}

\def\r{\rho}

\def\s{\sigma}

\def\t{\tau}

\def\o{\omega}

\def\D{\Delta}

\def\L{\Lambda}

\def\O{{\Omega}}
\def\Of{{\Omega_{\rm f}}}

\def\P{{\Phi}}

\def\T{\T}

\def\es{{\emptyset}}

\def\FF{{\mathcal F}}

\def\GG{{\mathcal{G}}}
\def\P{{\mathcal P}}
\def\FF{\boldsymbol{{\mathcal F}}}

\def\Ff{{\mathcal F}_{\rm f}}

\def\EE{{\mathcal E}}

\def\CC{{\mathcal C}}

\def\oo{\boldsymbol{\omega}}
\def\OO{\boldsymbol{\Omega}}
\def\xx{\boldsymbol{x}}
\def\yy{\boldsymbol{y}}

\def\PP{\boldsymbol{P}}

\def\V|{{\Vert}}


\begin{document}
\author{Benedikt Jahnel}
\address[Benedikt Jahnel]{Weierstrass Institute Berlin, Mohrenstr. 39, 10117 Berlin, Germany, \texttt{https://www.wias-berlin.de/people/jahnel/}}
\email{Benedikt.Jahnel@wias-berlin.de}

\author{Christof K\"ulske}
\address[Christof K\"ulske]{Ruhr-Universit\"at   Bochum, Fakult\"at f\"ur Mathematik, D44801 Bochum, Germany, \texttt{http://www.ruhr-uni-bochum.de/ffm/Lehrstuehle/Kuelske/kuelske.html}}
\email{Christof.Kuelske@ruhr-uni-bochum.de}

\title{Gibbsian representation for point processes via hyperedge potentials}

\date{\today}

\maketitle

\begin{abstract}
We consider marked point processes on the $d$-dimensional euclidean space, defined in terms of a quasilocal specification based on marked Poisson point processes. We investigate the possibility of constructing absolutely-summable Hamiltonians in terms of hyperedge potentials in the sense of Georgii et al~\cite{DeDrGe12}. These potentials are a natural generalization of physical multi-body potentials which are useful in models of stochastic geometry. 

We prove that such representations can be achieved, under appropriate locality conditions of the specification. As an illustration we also provide such potential representations for the Widom-Rowlinson model under independent spin-flip time-evolution.  

Our paper draws a link between the abstract theory of point processes in infinite volume, the study of measures under transformations, and statistical mechanics of systems of point particles.
\end{abstract}

\section{Introduction}
In this note we study models for not necessarily translation-invariant Poisson point processes (PPP) in euclidean space $\R^d$ with general marks. 
Such models are the subject in the infinite-volume statistical mechanics of classical point particles 
which interact via potentials. They are already very interesting when there are no marks (or internal states 
of particles), and only 
the positions of the colorless point particles are relevant. 
Potentials coming from physics are often pair potentials. Take as an example the 
famous Lennard-Jones potential. For results on existence of such models in the infinite volume, see \cite{Ru99,Ru70}.
Also more general potentials than pair potentials appear, 
describing interactions between finite collections of particles. 
These are quite relevant in physics as well,  
see for instance the proof of a phase transition for a long (but finite) range potential 
involving 4-body interactions in \cite{LeMaPr99}. 
For models from statistical physics with marks, see e.g.~the Potts gas in \cite{GeHa96}. 
The famous Widom-Rowlinson model (WRM) is a specific example for this 
which is proved to have a phase transition in the infinite volume \cite{Ru71,ChChKo95,GiLeMa95}.  

PPPs also have an interest which is independent from the physical motivation in statistical mechanics
in models of stochastic geometry e.g.~\cite{HiJaPaKe16,HiJaPa16}. 
In the development of the fundamentals of an
infinite-volume theory (existence, uniqueness, variational principle, $\dots$) also for such systems
an important step was made by \cite{DeDrGe12} in the introduction of the more general notion of a hyperedge potential, see \cite{DeDrGe12}. For such potentials 
one allows the energetic contribution of a finite subset of particles (hyperedges) to depend 
also on the other points in the cloud, but only up to a finite horizon. This relaxation of the strict 
locality requirement on the level of potentials incorporates many models from stochastic geometry. 
In this note we are aiming for absolutely-summable representations of abstractly given point processes as Gibbs fields in terms of such hyperedge potentials.  

To compare, let us recall the simpler situation in statistical mechanics on the lattice $\Z^d$, or more generally countable index sets, where
 the notion of a quasilocal specification is fundamental for the development of 
Gibbsian theory in its purest form, see \cite{Ge11}. 
An absolutely-summable potential defines 
a finite-volume Hamiltonian $H_{\L}$ in a finite volume $ \L \subset \Z^d$,  
which depends in a quasilocal way on the boundary condition outside of $\L$.

On the lattice, going from nice potentials to Gibbsian specifications in lattice statistical mechanics is straightforward, while 
the opposite is more difficult. However, Kozlov and Sullivan \cite{Ko74,Ko76,Su73} showed 
how one may construct potentials with various convergence properties. 
For systems of point particles already going from Hamiltonians to measures is more delicate,  
for the opposite direction partial results were obtained in \cite{Ko76} where a convergent 
 representation in terms of the (necessarily unique) vacuum potential was obtained, 
while uniform absolute convergence could not be provided. It is a main aim of our paper 
to show how uniform absolute convergence can indeed be achieved in the class of Georgii's hyperedge potentials. 

\medskip

The paper is organized as follows. Section~\ref{Gibbs point processes} contains the setup of Gibbs Point Processes. In Section~\ref{Sec3} we discuss the notion of hyperedge potentials in the sense of~\cite{DeDrGe12}, and formulate as our main general results Theorem~\ref{Representation_Absolute} and Theorem~\ref{Representation_Absolute_2}, on the absolute summability of a hyperedge potential. Before doing so, we put in place Theorem~\ref{Representation}, on which we will build up later, and Corollary~\ref{Representation_Range}. These concern the convergence of 
the vacuum potential, and its finite-range property under the assumption of the strict Markov property of the specification. Versions of the results~\ref{Representation} and~\ref{Representation_Range} were obtained for the first time in~\cite{Ko76}. In Section~\ref{Sec_WRM} we discuss the two-color WRM under independent spin-flip dynamics, see~\cite{JaKu16}. This model shows quite interesting Gibbs non-Gibbs transitions, depending on activities and time, where also full measure sets of bad point appear (see for example also the initial paper~\cite{EnFeHoRe02} on Gibbs non-Gibbs transitions for the lattice Ising model). We explain that in the Gibbsian regimes there is always a hyperedge potential which even has a uniform horizon, and which converges absolutely and uniformly. Finally, the proofs including further comments are provided in Section~\ref{Proofs}.

\subsection{Acknowledgement}
This work is dedicated to the memory of Professor Hans-Otto Georgii. 
Benedikt Jahnel thanks the Leibniz program 'Probabilistic methods for mobile ad-hoc networks' for the support.
Christof K\"ulske thanks the Weierstrass Institute for its hospitality.

\section{Gibbs point processes}\label{Gibbs point processes}
\subsection{Setup} We consider the euclidean space $\R^d$ with $d\ge1$ equipped with its Borel-$\s$-algebra. Let $\O$ denote the set of all \textit{locally finite subsets} of $\R^d$, that is, for $\o\in\O$ we have $|\o_\L|=\#\{\o\cap\L\}<\infty$ for all bounded sets $\L\subset\R^d$. 
The polish space $E$ equipped with its Borel-$\s$-algebra $\mathfrak{E}$ will play the role of a \textit{local state space} or in the language of point processes the \textit{mark space}. We write $\s_\o\in E^{\o}$ for the marks of a configuration $\o\in\O$. The \textit{marked configurations} $\oo=(\o,\s_\o)$ are locally finite subsets of $\R^d\times E$ and we
denote $\OO$ the set of all such marked configurations with $\o\in\O$. Conversely we call $\o\in\O$ the \textit{grey configuration} of $\oo\in\OO$. 
We equip $\OO$ with the $\s$-algebra $\FF$ which is generated by the counting variables $\OO\ni \oo\mapsto|\oo\cap(\L\times B)|$ for bounded and measurable $\L\Subset \R^d$ and $B\in\mathfrak{E}$, i.e.~$\FF=\s\big(\{\oo:\, |\o_\L|=n, \s_{\o_\L}\in B\}:\, n\in\N,\L\Subset\R^d, B\in\mathfrak{E}^n\big)$.
Further we denote by $\OO_\L$ the set of all marked configurations in the measurable set $\L\subset\R^d$ and equip it with the corresponding trace $\s$-algebra $\FF_\L$ of $\FF$ on $\OO_\L$. We write $f\in\FF_\L$ if $f$ is measurable w.r.t.~$\FF_\L$ and $f\in\FF^b_\L$ if $f$ is additionally bounded in the supremum norm $\Vert\cdot\Vert$. 

\subsection{Gibbs point processes for Poisson modifications} In this section we setup Gibbsian point processes via Poisson specifications along the lines of $\l$-specifications for models on fixed geometries as in \cite[Chapter 1]{Ge11}. 
For any $\L\subset\R^d$ and $\oo\in\OO$ we use the short-hand notation $\oo_\L$ to indicate that points in $\L^c=\R^d\sm\L$ are eliminated. With $\oo_\L\oo_\D$ we indicate the configuration which consists of the union of $\oo_\L$ and $\oo_\D$ and similar for grey configurations. Let us start by adapting the notion of pre-modifications from \cite[Definition 1.31]{Ge11} to the continuum setting. 
\begin{defn}[Pre-modification]
Let $h=(h_\L)_{\L\Subset\R^d}$ be a family of measurable functions $h_\L: \OO^*\to [0,\infty)$ with common domain $\OO^*\subset\OO$. Then $h$ is called a {\em $\, \OO^*$-pre-modification} if for all measurable $\L\subset\D\Subset\R^d$ and $\oo,\oo'\in \OO^*$,
$$h_\D(\oo_\L\oo_{\L^c})h_\L(\oo'_\L\oo_{\L^c})=h_\L(\oo_\L\oo_{\L^c})h_\D(\oo'_\L\oo_{\L^c}).$$
\end{defn}
As the prime example of a pre-modification consider the \textit{Boltzmann weight}
$$h_\L=e^{-H_\L}$$
where the \textit{Hamiltonian} $H_\L$ is given by 
\begin{equation}\label{Ham}
\begin{split}
H_\L(\oo)=\sum_{\eta\Subset\o:\, \eta\cap\L\neq\emptyset}\Phi(\eta,\oo).
\end{split}
\end{equation}
Here the \textit{potentials} $\Phi({\eta},\cdot):\OO^*\to(-\infty,\infty]$ are measurable functions w.r.t.~$\FF_{\eta}$, 
governing the interaction of marked particles at locations $\eta$. 
For example consider the Potts Gas \cite{GeHa96} with
\begin{equation}\label{Potential}
\Phi({\eta},\oo)=\d_{\eta=\{x,y\}}[\d_{\s_x\neq\s_y}\phi(x-y)+\psi(x-y)]
\end{equation}
for some measurable and even functions $\phi,\psi:\,\R^d\to]-\infty,\infty]$ which includes also the case of the Widom-Rowlinson model \cite{ChChKo95,WiRo70}.

The introduction of the domain $\OO^*$ of \textit{admissible boundary conditions} is necessary since in the continuum setting, due to the possible accumulation of points, even in simple models with infinite-range interactions, Hamiltonians and hence pre-modifications might not be well defined everywhere. More precisely, the sum in \eqref{Ham} is only well-defined for boundary conditions $\oo\in\OO^*_\L\subset\OO$ such that 
$$\sum_{\eta\Subset\o:\, \eta\cap\L\neq\emptyset}(-\Phi(\eta,\oo)\vee 0)<\infty.$$
Note that $\OO^*_\D\subset\OO^*_\L$ for $\D\supset\L$ since
\begin{equation*}
\sum_{\eta\Subset\o:\, \eta\cap\D\neq\emptyset}(-\Phi(\eta,\oo)\vee 0)\ge \sum_{\eta\Subset\o:\, \eta\cap\L\neq\emptyset}(-\Phi(\eta,\oo)\vee 0)
\end{equation*}
and hence it suffices to consider the common domain $\OO^*=\lim_{n\uparrow\infty}\OO^*_{\L_n}$ where $\L_n=[n/2,n/2]^d$ denotes the centered box of length $n\in\N$. We will give a proper and more general definition of potentials in Section~\ref{Sec3}.

The notion of a pre-modification can be used to describe a large class of \textit{specifications}. 
\begin{defn}[Specification]
A {\em $\,\OO^*$-specification} is a family of proper probability kernels $\g=(\g_\L)_{\L\Subset\R^d}$ where each $\g_\L(\cdot|\oo)$ is defined for all $\oo\in\OO^*$ with  $\g_\L(\OO^*|\oo)=1$ and additionally satisfies the following consistency condition. For all measurable $\L\subset\D\Subset\R^d$ and $\oo\in\OO^*$
$$\g_\D(\g_\L(d \oo'|\cdot)|\oo)=\g_\D(d \oo'|\oo).$$
\end{defn}

Let us denote by $\PP$
the (maybe non-stationary) \textit{marked Poisson point process} (PPP) on $(\OO,\FF)$ with \textit{intensity measure} $\mu(dx,du)=\nu(d x)F(d u|x)$. Here $\nu$ is a $\s$-finite measure on $\R^d$ which is equivalent to the Lebesgue measure on $\R^d$
and $F$ is a kernel from $\R^d$ to the set of $\s$-finite measures on $(E,\mathfrak{E})$.
By $\PP_\L$ we denote the restriction of $\PP$ to $\OO_\L$. For $f\in\FF^b$ we will often use the short hand notation $\int\PP(d\oo)f(\oo)=\PP f$.
For a given family of density functions $\r=(\r_\L)_{\L\Subset\R^d}$, defined on a set of Poisson measure one, probability kernels can be defined via
\begin{equation}\label{PoissonMod}
\begin{split}
\g^\r_\L(f|\oo_{\L^c})=\int\PP_\L(d\oo_{\L}) f( \oo_{\L} \oo_{\L^c}   ) \r_{\L}(\oo_{\L} \oo_{\L^c}).
\end{split}
\end{equation}

The following definition labels $\r$ a \textit{Poisson modification} if the associated $\g$ is a specification, similar to \cite[Definition 1.27]{Ge11}.

\begin{defn}[Poisson-modification]
Let $\r=(\r_\L)_{\L\Subset\R^d}$ be a family of measurable functions $\r_\L: \OO^*\to [0,\infty)$ with common domain $\OO^*\subset\OO$. Then, $\r$ is called a {\em $\, \OO^*$-Poisson modification} if the family of probability kernels $\g^\r=(\g^\r_\L)_{\L\Subset\R^d}$ given by \eqref{PoissonMod} is a $\, \OO^*$-specification. A $\, \OO^*$-\textit{Poisson modification} is called {\em positive} if for all $\oo\in\OO^*$ and $\L\Subset\R^d$ we have $\r_\L(\oo)>0$, it is called {\em vacuum positive} if only $\r_\L(\oo_{\L^c})>0$ holds.
\end{defn}
Note that under the PPP, the empty set in finite volumes with positive Lebesgue measure has positive mass. Hence, for all $\oo\in\OO^*$ also 
$\oo_{\L^c}\in\OO^*$ for all $\L\subset\R^d$. 
As an example note that the Poisson modification of the WRM is not positive but vacuum positive. 
For a $\, \OO^*$-pre-modification $h$ the normalization $Z_\L(\oo_{\L^c})=\int\PP_\L(d\oo_\L)h_\L(\oo_\L\oo_{\L^c})$ is referred to as the \textit{partition function}.
The conditions on pre-modifications give rise to Poisson modifications. This is the content of the following lemma.

\begin{lem}\label{PM}
Let $h$ be a $\, \OO^*$-\textit{pre-modification} with
$$0<Z_\L(\oo_{\L^c})<\infty$$
 for all $\L\Subset\R^d$ and $\oo_{\L^c}\in \OO^*$. Then $\r=(h_\L/Z_\L)_{\L\Subset\R^d}$ is a $\, \OO^*$-Poisson-modification if additionally for all $\D\Subset\R^d$ and $\oo\in\OO^*$ it holds that $\g_\D^\r(\OO^*|\oo)=1$.
Conversely, any $\, \OO^*$-\textit{modification} $\r$ is also a $\, \OO^*$-\textit{pre-modification}.
\end{lem}
Next we give a definition of Gibbs point processes via the DLR equation similar to the one for classical Gibbs measures on deterministic spatial graphs see \cite{Ge11}.
\begin{defn}[Gibbs point processes]
A random field $\P$ is called a {\em Gibbs point process} for the $\OO^*$-specification $\g$ iff for every $\L\Subset\R^d$ and for any $f\in\FF^b$, 
\begin{equation}\label{DLR}
\int\P(d\oo) f(\oo) = \int  \P(d\oo) \int \g_\L(d\oo'_\L|\oo)f(\oo'_\L\oo_{\L^{\rm c}})
\end{equation} 
and $\P(\OO^*)=1$. We denote the set of all such measures $\GG(\g)$.
\end{defn}
Existence of Gibbs point processes and the appearance of phase-transitions of multiple solutions to the so-called DLR equation \eqref{DLR} have been proved in a number of cases, see for example \cite{GeHa96,ChChKo95,DeDrGe12}.
In the next section we present our main result.

\section{Hyperedge potentials and the representation theorem}\label{Sec3}

\subsection{Hyperedge potentials}
Let us start by giving a more formal definition of interaction potentials in the continuum. 
For this let us denote by $\Of=\{\o\in\O:\, |\o|<\infty\}$ the set of finite configurations in $\O$ and $\Ff$ the trace $\s$-algebra of $\FF$ in $\Of$. The product space $\Of\times\OO$ carries the product $\s$-algebra $\Ff\otimes\FF$. With $\EE\subset\bar\EE=\{(\eta,\oo)\in\Of\times\OO:\, \eta\subset\o\}$ we denote a \emph{hypergraph structure} of $\OO$ as presented in \cite{DeDrGe12} for models with trivial single-site state-space. For $\oo\in\OO$ we write $\EE(\oo)=\{\eta\Subset\o:\, (\eta,\oo)\in\EE\}$.
Based on the hypergraph structure we now define \emph{hyperedge potentials}.
\begin{defn}[Hyperedge Potential]\label{Potential}
A hyperedge {\em potential} (or simply potential) is a 
measurable function $\Phi: \EE\mapsto(-\infty,\infty]$ 
with the following properties:
\begin{enumerate}
\item {\em Finite-horizon}: For each $(\eta,\oo)\in\EE$ there exists $\D(\eta,\oo)\Subset\R^d$ such that if $(\eta,\oo')\in\EE$ and $\oo_{\D(\eta,\oo)}=\oo'_{\D(\eta,\oo)}$, then $\Phi(\eta,\oo)=\Phi(\eta,\oo')$.
\item {\em Well-definedness}: For all $\L\Subset\R^d$ 
the series 
\begin{equation*}
\begin{split}
H_\L(\oo)=\sum_{\eta\in\EE(\oo):\, \eta\cap\L\neq\emptyset}\Phi(\eta, \oo)
\end{split}
\end{equation*}
exists in the sense that $H_\L(\oo)$ is the limiting point of the net 
$$\Big(H_{\L,\D}(\oo)\Big)_{\D\Subset\R^d}$$
with 
\begin{equation*}
\begin{split}
H_{\L,\D}(\oo)=\sum_{\eta\in\EE(\oo_\D):\, \eta\cap\L\neq\emptyset}\Phi(\eta,\oo).
\end{split}
\end{equation*}
\end{enumerate}
\end{defn} 
For $\L\Subset\R^d$ and $r>0$ we denote by $B_r(\L)=\{x\in\R^d:\, |x-y|<r \text{ for some }y\in\L\}$ the $r$-mollification of $\L$. Next we distinguish potentials in view of their finite-horizon properties. 
\begin{defn}[Uniform finite-horizon \& vacuum potentials]\label{Uniform finite-horizon}
We call a potential a
\begin{enumerate}
\item {\em uniformly finite-horizon potential} if for all $(\eta,\oo)\in\EE$ the finite-horizon property holds with $\D(\eta,\oo)=\D(\eta)$.
\item {\em $r$-uniformly finite-horizon potential} if for all $(\eta,\oo)\in\EE$ the finite-horizon property holds with $\D(\eta,\oo)=B_r(\eta)$ with $r>0$.
\item {\em vacuum potential} if for all $(\eta,\oo)\in\EE$ the finite-horizon property holds with $\D(\eta,\oo)=\eta$.
\end{enumerate}
\end{defn}
We can further distinguish different types of potentials w.r.t.~their convergence properties. In order to make the connection to the domains $\OO^*$ of admissible configurations, let us write $\EE^*$ for hypergraph structures which are subsets of $\bar \EE^*=\{(\eta,\oo)\in\Of\times\OO^*:\, \eta\subset\o\}$. 
\begin{defn}[Potential convergence]\label{PotConv}
We call a potential $\Phi$ on $\EE^*$
\begin{enumerate}
\item {\em uniformly convergent} if for all $\L\Subset\R^d$ we have 
$$\lim_{\D\uparrow\R^d}\sup_{\oo\in\OO^*}|H_{\L,\D}(\oo)-H_\L(\oo)|=0.$$
\item {\em absolutely summable} if for all $\L\Subset\R^d$ and $\oo\in\OO^*$ we have
$$\sum_{\eta\in \EE^*(\oo):\, \eta\cap\L\neq\emptyset}|\Phi(\eta,\oo)|<\infty.$$
\end{enumerate}
\end{defn}
Clearly, (1) implies well-definedness, but absolute summability does not imply uniform convergence due to the non-fixed geometry. The latter implication is correct for example for lattice systems, see~\cite[Chapter 2.1.]{Ge11}.
\subsection{Representation via vacuum potentials}
The next definition describes the fundamental goal behind this work.
\begin{defn}[Potential representation]
We say that a {\em potential $\Phi$ represents the $\OO^*$-Poisson modification $\r$} if $\Phi$ is defined on a hypergraph structure $\EE^*$ and for all $\L\Subset\R^d$ and $\oo\in\OO^*$ 
\begin{equation*}
\begin{split}
\r_\L(\oo)=Z_\L(\oo_{\L^c})^{-1}\exp\big(-\sum_{\eta\in\EE^*(\oo):\, \eta\cap\L\neq\emptyset}\Phi(\eta, \oo)\big).
\end{split}
\end{equation*}
\end{defn}
Our first result establishes existence of such potentials for given pre-modifications under the condition of vacuum positivity and continuity required to hold only in the direction of the vacuum. 
\begin{defn}[Vacuum quasilocality]
We call a real-valued measurable function $f$ with domain $\OO^*\subset\OO$ {\em vacuum quasilocal} if for all $\oo\in\OO^*$ we have that
$$\lim_{\L\uparrow\R^d}|f(\oo)-f(\oo_\L)|=0.$$
Moreover, $f$ is called {\em vacuum uniformly log-quasilocal} if $f$ is positive and 
$$\lim_{\L\uparrow\R^d}\sup_{\oo\in\OO^*}|\log f(\oo)-\log f(\oo_\L)|=0.$$
\end{defn}
Here and in the sequel, the limits should be understood as limits of nets on $\{\L:\L\subset\R^d\}$ ordered by inclusion. Clearly, uniform quasilocality w.r.t. the $\tau$-topology, i.e., 
$$\lim_{\L\uparrow\R^d}\sup_{\oo,\oo'\in\OO}|f(\oo_\L\oo'_{\L^c})-f(\oo)|=0$$
implies vacuum quasilocality but not uniform log-quasilocality even if $f$ is assumed positive. The last implication is true under the additional assumption of uniform positivity which is meaningful for example in lattice systems. But in our continuous setting even for $f$ given as the Poisson modification of the Potts gas we have uniform log-quasilocality but no uniform positivity. We call $\, \OO^*$-pre-modification $\r$ vacuum quasilocal (resp.~vacuum uniformly log-quasilocal) if for all measurable $\L\Subset\R^d$ we have that $\r_\L$ is vacuum quasilocal (resp.~vacuum uniformly log-quasilocal).

\begin{thm}\label{Representation}
Suppose $\r$ is a vacuum positive and vacuum quasilocal $\, \OO^*$-pre-modification such that for all $\L\Subset\R^d$ and $\oo_{\L^c}\in \OO^*$ we have $$\int\PP_\L(d \oo_\L)\r_\L(\oo_\L\oo_{\L^c})=1.$$ 
Then, there exists a unique vacuum potential $\Phi$ on the hypergraph structure $\bar\EE^*$. Moreover, if $\r$ is vacuum uniformly log-quasilocal, then $\Phi$ is uniformly convergent. 
\end{thm}
Let us note that for example in the lattice case, potentials can be constructed which are unique w.r.t.~certain $\a$-normalizations where $\a$ an arbitrary single-site measure. This freedom is not available in the continuum case since the geometry is not fixed. Further we note that the construction of the vacuum potential has been performed multiple times for lattices systems, see for example \cite{Gr73,Ko74}, and even more general point fields in \cite[Theorem 1B]{Ko76}, but there without the uniform convergence.

\medskip
The following statement that finite-range properties of Poisson modifications transfer to their associated vacuum potential is already partially presented in  \cite[Lemma 2]{Ko76}.
\begin{cor}\label{Representation_Range}
Let $\r$ be as in Theorem~\ref{Representation} and $\Phi$ the corresponding vacuum potential. Additionally assume that $\r$ is of range $r>0$, i.e., for all $\L\Subset\R^d$, $\r_\L$ is $\FF_{B_r(\L)}$ measurable. Then, if $\eta$ is such that there exist $x,y\in\eta$ with $|x-y|>r$, we have $\Phi(\eta,\oo)=0$.
\end{cor}
\subsection{Representation via absolutely-summable potentials}
The following main result of this paper shows that it is possible to derive a representation for vacuum uniformly log-quasilocal $\OO^*$-pre-modifications given by an absolutely-summable potential. This representation has the uniformly finite-horizon property but is no longer unique. 
\begin{thm}\label{Representation_Absolute}
Let $\r$ be a vacuum uniformly log-quasilocal $\, \OO^*$-pre-modification.
Then, there exists a representation of $\r$ via an absolutely-summable potential with uniform finite-horizon property.
\end{thm}
Let us note that the constructed absolutely-summable potential is defined on a much sparser hypergraph structure $\EE^*$ described in the remarks following the proof.

\subsection{Representation via translation-invariant potentials}
For $x\in\R^d$ we denote by $\theta_x:\O\to\O$ the spatial shift of configurations by $x$, i.e., $\theta_x\o=\{y\in\R^d:\, y-x\in \o\}$. With a slight abuse of notation we also write $\theta_x\oo=\theta_x(\o,\s_\o)=(\theta_x\o,\theta_x\s_{\theta_x\o})$ where for $y\in\theta_x\o$,  $\theta_x\s_y=\s_{{y-x}}$ and $\theta_x\L=\L+x$ for all $\L\Subset\R^d$. In this section, we assume that $\PP$ is translation invariant with intensity measure $\mu(d\s, d x)= \a\nu(d \s)d x$ where, $\a>0$ and $\nu$ is a finite measure on $(E,\mathfrak{E})$. A $\, \OO^*$-Poisson modification is called translation invariant, if for all $\oo\in\OO^*$ and $x\in\R^d$ we have $\theta_x\oo\in\OO^*$ and $\r_\L(\oo)=\r_{\theta_x\L}(\theta_x\oo)$ for all $\L\Subset\R^d$. A potential $\Phi$ is called translation invariant, if for all $(\eta,\oo)\in\EE^*$ and $x\in\R^d$ we have $(\theta_x\eta,\theta_x\oo)\in\EE^*$ and $\Phi(\theta_x\eta, \theta_x\oo)=\Phi(\eta, \oo)$.

\medskip
It is easy to see from the proof that, under the conditions of Theorem~\ref{Representation}, if the $\, \OO^*$-pre-modification $\r$ is additionally translation invariant, its unique vacuum potential is also translation invariant. However, the absolutely-summable potential constructed in the proof of Theorem~\ref{Representation_Absolute} is not translation invariant even if the associated $\, \OO^*$-pre-modification is translation invariant. This is a consequence of the resummation method based on a non-translation-invariant ordering of points, see Figure~\ref{PIX_2}. Our last theorem will improve on this aspect by using an additional summability condition in the quasilocality, see below. Let $\L_n\Subset\R^d$ denote the centered ball with diameter $n\in\N$.
\begin{defn}[Summable-vacuum-uniform log-quasilocality]
We call a real-valued measurable function $f$ with domain $\OO^*\subset\OO$ {\em summable-vacuum-uniformly log-quasilocal} if $f$ is positive and for 
$$\k(n)=\sup_{\oo\in\OO^*}|\log f(\oo)-\log f(\oo_{\L_n})|$$
we have $\sum_{n=1}^\infty n^d\k(n)<\infty$. 
\end{defn}
As presented in the next section, as an example, the element $\r_0$ (indexed by the origin) of the $\, \OO$-pre-modification $\rho$ of a time-evolved Widom-Rowlinson model from~\cite{JaKu16}, is summable-vacuum-uniformly log-quasilocal.

\begin{thm}\label{Representation_Absolute_2}
Let $\r$ be a translation-invariant $\, \OO^*$-pre-modification where $\r_0$ is summable-vacuum-uniformly log-quasilocal. Then on $\{\oo\in\OO:\, \sup_{n\in\N}n^{-d}|\o_{\L_n}|<\infty\}$, there exists a representation of $\r$ via a translation-invariant absolutely-summable potential with uniform finite-horizon property.
\end{thm}

\section{Potentials for a time-evolved Widom-Rowlinson model}\label{Sec_WRM}
As an illustration we consider the WRM under independent spin flip as presented in \cite{JaKu16}. We start by recalling the model. 
\subsection{The WRM under independent spin flip}
The WRM, as initially proposed in \cite{WiRo70}, is a hard-core repulsion model with single spin space $E=\{+,-\}$ and $\OO$-Poisson-modification given by
$$\chi_\L(\oo)=\one\{\text{for all }x,y\in \o\text{ with }|x-y|<2r:\,\s_x=\s_y\}.$$
Alternatively, it can be described via the vacuum potential 
\begin{equation*}
\Phi(\eta,\oo)=\infty\times \one_{|x-y|< 2r}\one_{\eta=\{x,y\}}\one_{\s_x\neq\s_y}.
\end{equation*}
Note that the interaction is of range $2r>0$. 
The underlying PPP is given by the superposition of two PPP with spatially homogeneous intensities $\l_+\ge\l_->0$, in particular, the model is translation invariant. It is well known, see for example \cite{ChChKo95,Ru71,GiLeMa95}, that the symmetric WRM with $\l_+=\l_-$ exhibits a phase-transition in the high-intensity regime. Writing $\OO^{\rm f}=\{\oo\in\OO:\, \o \text{ has no infinite cluster}\}$, high-intensity here means that for large enough $\l_++\l_-$, the WRM is concentrated on $\OO\sm\OO^{\rm f}$. We speak of low-intensity if the WRM is concentrated on $\OO^{\rm f}$. A cluster $C\subset\o$ is a maximally connected subset of $\o$, i.e., a set $C$ such set for all $x,y\in C$ there exists a set of points $\{z_1,\dots,z_{n}\}\subset C$ with $|x-z_{1}|<2r$, $|z_n-y|<2r$ and $|z_i-z_{i+1}|<2r$ for all $1\le i\le n$. For the purpose of the present overview we focus only on the maximal measure $\mu^+$ obtained as a monotone limit of specifications with full plus boundary conditions along an exhausting volume sequence, also in regimes of non-uniqueness of the Gibbs measure.

\medskip
The dynamics is given by rate-one Poisson flips independently attached to every particle, i.e., the probability, that a site in the plus-state, is in the plus-state at time $t\ge0$ is given by 
$$p_t(+,+)=\tfrac{1}{2}(1+e^{-2t})$$
with $p_t(+,-)$, $p_t(-,-)$ and $p_t(-,+)$ defined accordingly. The main findings of \cite{JaKu16} are that, depending on asymmetry of the WRM and time, there is a sharp Gibbs-non-Gibbs transition in the sense that the time-evolved WRM can be described as a Gibbs measure for a quasilocal $\OO$-Poisson modification (respectively $\OO^{\rm f}$-Poisson modification) $\r$ or not. Quasilocality here is defined as continuity w.r.t.~the $\t$-topology, which is defined via convergence tested on all measurable local functions, and not just continuous functions.  Focussing on the asymmetric case $\l_+>\l_-$ with initial WRM being in the plus extremal state, 
the critical time $t_G$ is given by the unique positive solution of 
$$b:=\frac{\l_-p_t(+,+)}{\l_+p_t(+,-)}=1.$$
A set of configurations where discontinuities of finite-volume conditional probabilities of the time-evolved measure can not appear at the critical time $t_G$, can be defined by $\OO^+=\{\oo\in\OO:\, \o \text{ has no infinite cluster }C\text{ with } \liminf_{n\uparrow\infty}|C\cap \L_n|^{-1}\sum_{x\in C\cap \L_n}\s_x\le 0\}$. 
In Table~\ref{table1} we summarize the results for the asymmetric model, which is the most interesting case; see also~\cite{DeHo18} for recent results on the corresponding equilibrium model. In the table, when we write ``no quasilocal Poisson modification'' we mean that there exists no $\OO'\subset\OO$ such that the time-evolved WRM would be concentrated on $\OO'$ and there exists a quasilocal $\OO'$-Poisson modification. In all other cases, the quasilocal Poisson modification can be constructed explicitly and will be introduced in the following section.
\makegapedcells
\begin{table}[h!] 
  \centering
  \caption{Quasilocality (ql) transitions of Poisson-modifications for the time-evolved asymmetric WRM.}
\label{table1}
\scalebox{1}{
  \begin{tabular}{cV{3}c|c}
 time & high intensity & low intensity\\
\hlineB{5}
$0<t< t_G$  & no ql Poisson modification & ql $\OO^{\rm f}$-Poisson modification\\ \cline{1-3}
$t=t_G$  & ql $\OO^+$-Poisson modification & ql $\OO^{\rm f}$-Poisson modification\\ \cline{1-3}
$t_G<t\le\infty$  & ql $\OO$-Poisson modification& ql $\OO$-Poisson modification\\
\end{tabular}
}
\end{table}

\subsection{Uniformly finite-horizon potentials for the time-evolved WRM}
It is one of the nice features of the time-evolved WRM that Poisson modifications can be explicitly constructed. We now use the time-evolved two-color PPP as the a-priori measure, in other words, the underlying point process $\PP$ is now given by the superposition of two PPP with spatially homogeneous intensity measure
\begin{equation*}
\begin{split}
F(\{+\})=\l_+p_t(+,+)+\l_-p_t(-,+)\cr
F(\{-\})=\l_+p_t(+,-)+\l_-p_t(-,-).
\end{split}
\end{equation*}
The $\OO$-Poisson modification (respectively $\OO^{\rm f}$-Poisson modification, $\OO^+$-Poisson modification) $\r$ is given by 
\begin{equation*}
\begin{split}
\r_\L(\oo_\L\oo_{\L^c})=h_\L(\oo_\L\oo_{\L^c})/\PP_\L (h_\L)(\oo_{\L^c})
\end{split}
\end{equation*}
where
\begin{equation*}
\begin{split}
h_\L(\oo)&=
\frac{1}{(1+a)^{|\oo_{\L}|^+}(1+b)^{|\oo_{\L}|^-}}\prod_{C\in\CC^{\rm f}_\L(\o)}(1+a^{|\oo_C|^+}b^{|\oo_C|^-}).
\end{split}
\end{equation*}
Here we used the following notation: $a=\l_-p_t(+,-)/(\l_+p_t(+,+))$; $|\oo_C|^\pm$ denotes the number of plus (respectively minus) spins in $\oo_C$; $\CC^{\rm f}_\L(\o)$ (respectively $\CC^{\infty}_\L(\o)$) denotes the set of clusters in $\o$ with nonempty intersection with the volume $\L$ and which are finite (respectively infinite). Let us note that in the original presentation of the model in~\cite{JaKu16}, we do not introduce a representation via Poisson modifications with respect to the two-color Poisson process $\PP$, like we do in the present setup. This might make it difficult for the reader to make the connection, but it is a straight-forward (though cumbersome) computation to derive the present representation. Let us mention here at least that in~\cite{JaKu16} we use the notation $\r(\oo_C)=a^{|\oo_C|^+}b^{|\oo_C|^-}$. 

\medskip
In order to arrive at a potential representation for $\r$, we write $\CC^{\rm f}(\o)$ for the set of all finite clusters in $\o$ and compute
\begin{equation*}
\begin{split}
\log h_\L(\oo_\L)=\sum_{C\in\CC^{\rm f}(\o_\L)}\log(1+a^{|\oo_C|^+}b^{|\oo_C|^-})-|\oo_\L|^+\log(1+a)-|\oo_\L|^-\log(1+b).
\end{split}
\end{equation*}
Note that the second and third summand on the r.h.s.~form single-site potentials which can be considered as part of the a-priori measure by incorporating them into the mark distribution $F$. Only the first summand describes interactions.
Hence, we can define a potential
$$\Psi(\eta,\oo)=\log(1+a^{|\oo_{\eta}|^+}b^{|\oo_{\eta}|^-})\one_{\eta\in\CC^{\rm f}(\o)}$$
which is $2r$-uniformly finite-horizon potential is the sense of Definition~\ref{Uniform finite-horizon}. To see this, note that $\Psi$ assigns an interaction energy to finite clusters of $\o$. In order to decide wether a subset $\eta\Subset\o$ is a cluster, it suffices to know $\o_{B_{2r}(\eta)}$. Note that $a<1$ by definition. In the low-intensity regime $\PP(\OO^{\rm f})=1$ and the number of clusters attached to any finite volume is finite. Hence the $2r$-uniformly finite-horizon Hamiltonian defined via $\Psi$ exists in the domain $\OO^{\rm f}$. Existence is also guaranteed at the critical time on the domain $\OO^+$ since then $b=1$ and $\Psi$ decays exponentially as the cluster size grows unless the number of plus spins is macroscopically vanishing, but $\PP(\OO^+)=1$. At supercritical times and high-intensities, also $b<1$ and hence we have exponential decay. Finally, we note that $\Psi$ is defined on the hypergraph structure $\EE_\CC=\{(\eta,\oo)\in\bar\EE:\, \eta=\CC^{\rm f}(\o)\}$ of finite clusters.

\subsection{The vacuum potential for the time-evolved WRM}
In the following we derive the vacuum potential representation and investigate its decay properties. Using the definition in equation \eqref{VacPot} we have 
\begin{equation*}
\begin{split}
\Phi({\eta},\oo)&=- \sum_{\xi\subset\eta} (-1)^{|\eta\sm\xi|}\log\frac{\r_\L(\oo_{\xi})}{\r_\L(\es_\L)}=- \sum_{\xi\subset\eta} (-1)^{|\eta\sm\xi|}\sum_{C\in\CC(\xi)}\Psi(C,\oo).
\end{split}
\end{equation*}
\begin{lem}\label{VacPotNonZero}
The vacuum potential $\Phi({\eta},\oo)$ is non-zero only if $\eta$ is a cluster.
\end{lem}
In particular $\Phi$ is again defined on $\EE_\CC$ and if $\Phi({\eta},\oo)\neq0$ then
\begin{equation*}
\begin{split}
\Phi({\eta},\oo)&=- \sum_{\xi\subset\eta} (-1)^{|\eta\sm\xi|}\Psi(\xi, \oo).
\end{split}
\end{equation*}
Note that clusters can become infinitely long with positive probability in the high intensity regime, for details see \cite{JaKu16}. Further, note that spatial positioning inside clusters do not play any r\^ole in $\Psi(\eta,\oo)$. Hence, for nonzero $\Phi$, we can write
\begin{equation*}
\begin{split}
\Phi({\eta},\oo)&=- \sum_{k=0}^{|\oo_\eta|^+} \binom{|\oo_\eta|^+}{k}(-1)^{|\oo_\eta|^+-k}\sum_{l=0}^{|\oo_\eta|^-} \binom{|\oo_\eta|^-}{l}(-1)^{|\oo_\eta|^--l}\k(k,l)
\end{split}
\end{equation*}
where $\k(k,l)=\log(1+a^kb^l)$.
Expanding the logarithm yields, 
\begin{equation*}
\begin{split}
\Phi({\eta},\oo)&=(-1)^{|\eta|+1}\sum_{j=1}^\infty(-1)^j\frac{1}{j}(1-a^j)^{|\oo_\eta|^+}(1-b^j)^{|\oo_\eta|^-}.
\end{split}
\end{equation*}
The vacuum potential is expected to converge slowly and is not absolutely summable. Let us conclude the discussion by the following (non-optimal) upper bound for the critical-time case where $\Phi({\eta},\oo)=0$ if $|\oo_\eta|^->0$ and hence $\Phi$ is given by
$$\phi(n)=(-1)^{n+1}\sum_{j=1}^{\infty} (-1)^{j}\frac{1}{j}(1-\a^{j})^{n}\qquad\text{with }\a=(\l_-/\l_+)^2<1.$$
\begin{lem}\label{AsymptoticPot}
We have that $\limsup_{n\uparrow\infty}|\phi(n)|\log n\le C$ for some $C>0$.
\end{lem}

\subsection{Existence of absolutely-summable potentials} The time-evolved WRM in the Gibbsian regime $t>t_G$ is uniformly quasilocal in the $\t$-topology and hence also uniformly vacuum quasilocal, see \cite{JaKu16}. The next result shows in particular, that it is also vacuum uniformly log-quasilocal. 
\begin{prop}\label{AbsSumWRM}
In the WRM under independent spin flip in the regime where the associated $\OO$-pre-modification $\r$ is quasilocal, $\r$ is even uniformly log-quasilocal. 
\end{prop}
As a consequence of Theorem~\ref{Representation_Absolute} the time-evolved WRM in the corresponding regime can thus be written as a Gibbs measures w.r.t.~an absolutely-summable potential. The result of Proposition~\ref{AbsSumWRM} also holds at the critical time $t_G$, but we do not prove it here. However, in this case, already the vacuum potential is absolutely summable. 

\subsection{Existence of absolutely-summable translation-invariant potentials} The time-evolved WRM is translation-invariant. In the Gibbsian regime $t>t_G$, it enjoys a representation by a $\OO$-pre-modification which is even summable-vacuum-uniformly log-quasilocal. 
\begin{prop}\label{AbsSumTransWRM}
In the WRM under independent spin flip in the regime where the associated $\OO$-pre-modification $\r$ is quasilocal, $\r_0$ is summable-vacuum-uniformly log-quasilocal. 
\end{prop}
Thus, on the set $\OO'=\{\oo\in\OO:\, \sup_{n\in\N}n^{-d}|\o_{\L_n}|<\infty\}$, the time-evolved WRM $\mu_t$ can be also represented by an absolutely-summable potential which is translation invariant. But, for any $t>0$, $\mu_t(\OO')=1$ and thus $\rho$, considered as a $\OO'$-pre-modification has a corresponding absolutely-summable translation-invariant potential representation. To see that $\mu_t(\OO')=1$, first note that $\mu_t(\OO')=\mu_0(\OO')=1$ and by Borel-Cantelli, it suffices to verify that $\sum_{n=0}^\infty\mu_0(|\o_{\L_n}|\ge cn^d)<\infty$ for some sufficiently large $c>0$. But, 
\begin{equation*}
\begin{split}
\mu_0(|\o_{\L_n}|\ge cn^d)=\int\mu_0(d\oo)\g^\chi_{\L_n}(|\o_{\L_n}|\ge cn^d|\oo)\le\PP(|\o_{\L_n}|\ge cn^d)\exp(\tfrac{\pi(\l_++\l_-)}{2^d}n^d)
\end{split}
\end{equation*}
and thus, using Chernov bounds for the Poisson random variable $|\o_{\L_n}|$, for sufficiently large $c>0$, we have that $\PP(|\o_{\L_n}|\ge cn^d)\exp(\pi(\l_++\l_-)n^d/2^d)\le \exp(-n^d)$, which is summable in $n$.

\section{Proofs}\label{Proofs}
\begin{proof}[Proof of Lemma~\ref{PM}]
First note that $\PP_\L \r_\L=1$ and in particular for $f\in\FF^b_{\L^c}$ and $\oo_{\L^c}\in\OO^*$ we have $\g^\r_{\L}(f|\oo_{\L^c})=f(\oo_{\L^c})$ and hence $\g^\r$ is proper. 
As for the consistency, note that using the concentration on $\OO^*$ we have that
\begin{equation*}
\begin{split}
&\g^\r_\D(\g^\r_\L(f|\cdot)|\oo_{\D^c})\cr
&=\int\PP_\D(d\oo'_\D)\r_\D(\oo'_\D\oo_{\D^c})\int\PP_\L(d\oo''_\L)\r_\L(\oo''_\L\oo'_{\D\sm\L}\oo_{\D^c})f(\oo''_\L\oo'_{\D\sm\L}\oo_{\D^c})\cr
&=\int\PP_\D(d\oo'_\D)\frac{h_\D(\oo'_\D\oo_{\D^c})}{\PP_\D h_\D(\oo_{\D^c})}\int\PP_\L(d\oo''_\L)\frac{h_\L(\oo''_\L\oo'_{\D\sm\L}\oo_{\D^c})}{\PP_\L h_\L(\oo'_{\D\sm\L}\oo_{\D^c})}f(\oo''_\L\oo'_{\D\sm\L}\oo_{\D^c})\cr
&=\int\PP_\D(d\oo'_\D)\frac{h_\L(\oo'_\D\oo_{\D^c})}{\PP_\D h_\D(\oo_{\D^c})}\int\PP_\L(d\oo''_\L)\frac{h_\D(\oo''_\L\oo'_{\D\sm\L}\oo_{\D^c})}{\PP_\L h_\L(\oo'_{\D\sm\L}\oo_{\D^c})}f(\oo''_\L\oo'_{\D\sm\L}\oo_{\D^c})\cr
&=\int\PP_{\D\sm\L}(d\oo'_\D)\int\PP_\L(d\oo''_\L)\frac{h_\D(\oo''_\L\oo'_{\D\sm\L}\oo_{\D^c})}{\PP_\D h_\D(\oo_{\D^c})}f(\oo''_\L\oo'_{\D\sm\L}\oo_{\D^c})=\g^\r_\D(f|\oo_{\D^c})
\end{split}
\end{equation*}
where we used the pre-modification property in the forth line.

\medskip
For the converse statement, note that by the above calculations, the equation 
\begin{equation*}
\begin{split}
\r_\D(\oo_\D\oo_{\D^c})=\r_\L(\oo_\D\oo_{\D^c})\int\PP_\L(d\oo'_\L)\r_\D(\oo'_\L\oo_{\D\sm\L}\oo_{\D^c})
\end{split}
\end{equation*}
must be satisfied for $\PP_\D$ almost all $\oo_\D$. Thereby, using the above equation twice,
\begin{equation*}
\begin{split}
\r_\D(\oo_\D\oo_{\D^c})\r_\L(\oo''_\L\oo_{\D\sm\L}\oo_{\D^c})&=\r_\L(\oo_\D\oo_{\D^c})\r_\L(\oo''_\L\oo_{\D\sm\L}\oo_{\D^c})\int\PP_\L(d\oo'_\L)\r_\D(\oo'_\L\oo_{\D\sm\L}\oo_{\D^c})\cr
&=\r_\L(\oo_\D\oo_{\D^c})\r_\D(\oo''_\L\oo_{\D\sm\L}\oo_{\D^c})\cr
\end{split}
\end{equation*}
for $\PP_\L\times\PP_\L$ almost all $\oo_\L,\oo''_{\L}$, which concludes the proof.
\end{proof}

\begin{proof}[Proof of Theorem~\ref{Representation}] We prove in several steps similar to \cite[Theorem 2.30]{Ge11} and \cite[Theorem 1]{Ko76}. 
We claim, that the potential is given by 
\begin{equation}\label{VacPot}
\begin{split}
\Phi(\eta,\oo)=
\begin{cases} 
- \sum_{\xi\subset\eta} (-1)^{|\eta\sm\xi|}\log\frac{\r_\L(\oo_{\xi})}{\r_\L(\es_\L)} &\mbox{if } \r_\L(\oo_\eta)>0 \\ 
+\infty & \mbox{if } \r_\L(\oo_\eta)=0 
\end{cases} 
\end{split}
\end{equation}
where the definition is independent of $\L$ as long as $\eta\subset\L$. 

\medskip
{\bf Step 1:} Note first, that by vacuum positivity $\r_\L(\es_\L)>0$. Further, for $\xi\subsetneq\eta\subset\L$ there exists $\D\subset\L$ such that with $\eta\sm \xi\subset\D\subset\L$ and $\xi\cap\D=\es$. Then, by the vacuum positivity assumption $\r_\D(\oo_{\xi})>0$ and the pre-modification property 
\begin{equation*}\label{PosErb}
\begin{split}
\r_\D(\oo_\eta)\r_\L(\oo_{\xi})=\r_\D(\oo_{\xi})\r_\L(\oo_\eta)
\end{split}
\end{equation*}
we thus have that $\r_\L(\oo_\eta)>0$ implies $\r_\L(\oo_{\xi})>0$ and hence $\Phi$ is well-defined.

\medskip
{\bf Step 2:} The potential $\Phi$ has the following properties.

(1) $\Phi(\eta,\cdot)$ is $\FF_\eta$ measurable since the evaluation is only w.r.t.~$\eta$.

(2) By the inclusion-exclusion principle we have 
$$\log\frac{\r_\L(\oo_\L)}{\r_\L(\es_\L)}=-\sum_{\eta\subset\o_\L}\Phi({\eta},\oo).$$

(3) $\Phi$ is vacuum normalized in the sense that for all $\xi\subsetneq\eta$
$\Phi(\eta,\oo_{\xi})=0$. Indeed let $\xi\subsetneq\eta$, then 
\begin{equation*}
\begin{split}
-\Phi(\eta,\oo_\xi)&=\sum_{\zeta\subset\xi}\sum_{\zeta'\subset\eta\sm\xi}(-1)^{|\eta\sm(\zeta\cup\zeta')|}\log\frac{\r_\L(\oo_\zeta)}{\r_\L(\es_\L)}\cr
&=\sum_{\zeta\subset\xi}(-1)^{|\xi\sm\zeta|}\log\frac{\r_\L(\oo_\zeta)}{\r_\L(\es_\L)}\sum_{\zeta'\subset\eta\sm\xi}(-1)^{|(\eta\sm\xi)\sm\zeta'|}
\end{split}
\end{equation*}
which is zero since $\sum_{\zeta'\subset\eta\sm\xi}(-1)^{|(\eta\sm\xi)\sm\zeta'|}=0$. 

\medskip
{\bf Step 3:} Next we show that the definition of $\Phi$ is independent of the volume $\L$ via the pre-modification property of $\r_\L$. For this, let $\es\neq\eta\subset\L'\subset\L$, then 
\begin{equation*}
\begin{split}
\sum_{\xi\subset\eta} (-1)^{|\eta\sm\xi|}\log\frac{\r_{\L}(\oo_{\xi})}{\r_{\L'}(\oo_{\xi})}=\log\frac{\r_{\L}(\es)}{\r_{\L'}(\es)}\sum_{\xi\subset\eta} (-1)^{|\eta\sm\xi|}=0.
\end{split}
\end{equation*}
Together with (3), this shows that $\Phi$ is a vacuum potential.

\medskip
{\bf Step 4:} For the existence of the Hamiltonian, note that formally
\begin{equation*}
\begin{split}
H_{\L,\D}(\oo)&=\sum_{\es\neq\eta\subset\o_\D}\Phi({\eta},\oo)-\sum_{\es\neq\eta\subset\o_{\D\setminus\L}}\Phi({\eta},\oo)\cr
&=\log\frac{\r_{\D\sm\L}(\oo_{\D\sm\L})}{\r_{\D\sm\L}(\es_{\D\sm\L})}-\log\frac{\r_{\D}(\oo_{\D})}{\r_{\D}(\es_\D)}=\log\frac{\r_{\D}(\oo_{\D\sm\L})}{\r_\D(\oo_\L\oo_{\D\sm\L})}=-\log\frac{\r_\L(\oo_\L\oo_{\D\sm\L})}{\r_{\L}(\oo_{\D\sm\L})}
\end{split}
\end{equation*}
where we used the pre-modification property twice. By vacuum positivity $\r_{\L}(\oo_{\D\sm\L})>0$ and thus $H_{\L,\D}(\oo)$ is well defined. Now, by assumption of vacuum quasilocality, as $\D$ tends to $\R^d$, we have 
$$H_\L(\oo_\L\oo_{\L^c})=-\log\frac{\r_\L(\oo_\L\oo_{\L^c})}{\r_{\L}(\oo_{\L^c})}.$$ 
Moreover if $\r_\L$ is vacuum uniformly log-quasilocal, it is in particular positive and we have 
\begin{equation*}
\begin{split}
\sup_{\oo\in\OO^*}&|H_{\L,\D}(\oo)-\log\frac{\r_{\L}(\oo_{\L^c})}{\r_\L(\oo)}|\le \sup_{\oo\in\OO^*}(|\log\frac{\r_{\L}(\oo_{\D\sm\L})}{\r_\L(\oo_{\D\sm\L}\oo_{\D^c})}|+|\log\frac{\r_{\L}(\oo_{\D})}{\r_\L(\oo_\D\oo_{\D^c})}|)
\end{split}
\end{equation*}
which tends to zero as $\D$ tends to $\R^d$. 

\medskip
{\bf Step 5:} Note that $h^\Phi_\L(\oo)=\exp(-H_\L(\oo))=\r_\L(\oo)/\r_\L(\oo_{\L^c})$ and the normalization is given by $\int\PP_\L(d\oo_\L) h^\Phi_\L(\oo_\L\oo_{\L^c})=1/\r_\L(\oo_{\L^c})$. Hence, $\r^\Phi=\r$.

\medskip
{\bf Step 6:} Finally, for the uniqueness, let $\Phi'$ be another vacuum potential with $\r^\Phi=\r^{\Phi'}$. Then, $\Phi'-\Phi$ is again a vacuum potential which is equivalent to zero in the sense that 
\begin{equation*}
\begin{split}
H^{\Phi'-\Phi}_\L=\log(h_\L^{\Phi'}/h_\L^\Phi)=\log(Z^{\Phi'}_\L/Z^\Phi_\L)
\end{split}
\end{equation*}
is measurable w.r.t.~$\FF_{\L^c}$. Then, it suffices to show that $\Psi=\Phi'-\Phi=0$. But for all $\L\Subset\R^d$ by the inclusion-exclusion principle,
$$\Psi(\eta,\oo)=\sum_{\es\neq\xi\subset\eta}(-1)^{|\eta\sm\xi|}H^{\Psi}_\L(\oo_\xi)=H^{\Psi}_\L(\es)\sum_{\es\neq\xi\subset\eta}(-1)^{|\eta\sm\xi|}=0$$
where we used the normalization in the last equation.
\end{proof}

\begin{proof}[Proof of Corollary~\ref{Representation_Range}]
Let $\eta$ be such that there exist $x,y\in\eta$ with $|x-y|=s>r$. Denote $\eta'=\eta\sm\{x,y\}$ and $B_r(x)$ the open ball with radius $r$ centered at $x\in\R^d$. Then, using the pre-modification property, we have 
\begin{equation*}
\begin{split}
-\Phi({\eta},\oo)&=\sum_{\xi\subset\eta} (-1)^{|\eta\sm\xi|}\log\frac{\r_\L(\oo_{\xi})}{\r_\L(\es_\L)}\cr
&=\sum_{\xi_1\subset\eta'}(-1)^{|\eta'\sm\xi_1|}[\log\frac{\r_\L(\oo_{\xi_1}\xx\yy)}{\r_\L(\oo_{\xi_1}\yy)}+\log\frac{\r_\L(\oo_{\xi_1})}{\r_\L(\oo_{\xi_1}\xx)}]\cr
&=\sum_{\xi_1\subset\eta'}(-1)^{|\eta'\sm\xi_1|}[\log\frac{\r_{B_{s-r}(x)}(\oo_{\xi_1}\xx \yy)}{\r_{B_{s-r}(x)}(\oo_{\xi_1}\yy)}-\log\frac{\r_{B_{s-r}(x)}(\oo_{\xi_1}\xx)}{\r_{B_{s-r}(x)}(\oo_{\xi_1})}]
=0
\end{split}
\end{equation*}
as required.
\end{proof}

The key to improve the possibly very poor summability properties of the vacuum potential is to apply a suitable resummation procedure. 
For the lattice such resummations have been used for the first time in \cite{Ko74} to improve convergence. It is interesting to note that resummations could even be used in certain cases of non-Gibbsian lattice systems, namely for the joint measures of quenched random systems.  
Here one obtains at least weakly Gibbsian representations. Having a weakly Gibbsian representation means that the Hamiltonians converge absolutely at least on a measure one set of configurations, but possibly not everywhere, see \cite{Ku01}.
In the continuum such resummations have not been done so far, to our knowledge. We will explain now, how nice they can be done, and how well indeed it works together with the notion of Georgii's hyperedge potential, see \cite{DeDrGe12}, as collected interactions can be naturally indexed with hyperedges when one allows an additional dependence up to a finite horizon.

\begin{proof}[Proof of Theorem~\ref{Representation_Absolute}]

Let us start by considering for every $x\in\R^d$ a co-final sequence $(\D_{x,m})_{m\ge1}$ of finite subsets in $\R^d$ to be specified later. Next, let $\ge$ denote a total ordering on $\R^d$ for which every locally finite subset has a least element. For example think of the cycling order where points are ordered first by their euclidean distance to the origin and then by their angles. Let $\L_x=\{y\ge x\}$ and define $A_{x,m}=\D_{x,m}\cap\L_x$ with $A_{x,0}=\es$ the part of the sequence such that the $x$ is the left endpoint. 
For $\eta\Subset\R^d$ we will write $l(\eta)$ and $r(\eta)$ to denote the left and right end points of $\eta$ in the given ordering. Further we define
$$P_{x,m}=\{\eta\Subset\R^d:\, l(\eta)=x \text{ and } r(\eta)\in(A_{x,m}\sm A_{x,m-1})\}$$
the set of finite subsets of $\R^d$ with left end point equal to $x$ and right end point in the $m$-annulus of the sequence $A_{x,m}$. 
Note that in particular, $\bigcup_{x,m}P_{x,m}=\{\eta:\, \eta\Subset\R^d\}$ is a disjoint partition of the set of finite subsets of $\R^d$. This is a certain grading of the set of finite subsets of $\R^d$. 

Now we perform the regrouping w.r.t.~the unique vacuum potential $\Phi$. For a given $\o\in\O$ and any $x\in\o$ we denote by $P^\o_{x,m}=P_{x,m}\cap\{\eta\Subset\o\}$ the set of subsets of $\o$ in the grading $P_{x,m}$.
Such a $P^\o_{x,m}$ might very well be empty. Note that $P^\o_{x,m}=P^{\tilde\o}_{x,m}$ if $\o_{A_{x,m}}=\tilde\o_{A_{x,m}}$. 
Next, we let $\o_{x,m}=\o\cap (\{x\}\cup (A_{x,m}\sm A_{x,m-1}))$ be the union of finite subsets of $\o$ which have $x$ as their left endpoint and all their other points lying in the $m$-annulus. In these sets we will accumulate the energy contribution of all $\eta\subset\o_{x,m}$.
In case $P^\o_{x,m}=\es$ we do not need such a representative as will become clear in the following definition. 
For $(\eta,\oo)\in\EE$ we define $\Psi(\eta,\oo)=0$ unless $\eta=\o_{x,m}$ for some pair $(x,m)$ in which case we put
\begin{equation*}
\Psi({\o_{x,m}},\oo)=\sum_{\eta\in P^\o_{x,m}}\Phi(\eta,\oo).
\end{equation*}
In words, the energy of vacuum interaction potentials within a class is accumulated in one interaction for each class, see Figure~\ref{PIX_2} for an illustration. The sum can contain configurations $\eta$ with points in $A_{x,m-1}$.
Clearly $\Psi$ is not a vacuum potential.
However, note that we have $\Psi({\o_{x,m}},\oo)=\Psi({\o_{x,m}},\tilde\oo)$ if $\o_{A_{x,m}}=\tilde\o_{A_{x,m}}$ and thus $\Psi$ has the finite-horizon property.
\begin{figure}[!htpb]
\centering
\begin{tikzpicture}[scale=0.85]
\draw[black,dashed] (0,0) circle (5);
\fill[blue] (-2.051,-3.967) circle (2pt);
\fill[blue] (4.473,0.222) circle (2pt);
\fill[blue] (-1.010,-2.549) circle (2pt);
\fill[blue] (-2.333,-3.153) circle (2pt);
\fill[blue] (4.749,0.014) circle (2pt);
\fill[blue] (3.907,-2.984) circle (2pt);
\fill[blue] (-1.607,-1.437) circle (2pt);
\coordinate[label=0: {${x_4}$}] (x_4) at (-1.607,-1.437);
\fill[blue] (-2.116,2.999) circle (2pt);
\fill[blue] (2.714,1.797) circle (2pt);
\fill[blue] (-2.505,0.040) circle (2pt);
\coordinate[label=0: {${x_6}$}] (x_6) at  (-2.505,0.040);
\fill[blue] (-3.808,-1.164) circle (2pt);
\fill[blue] (-3.021,1.309) circle (2pt);
\fill[blue] (3.370,-1.964) circle (2pt);
\fill[blue] (3.217,-0.957) circle (2pt);
\fill[blue] (-4.535,0.206) circle (2pt);
\fill[black!10!white] (0.701,1.588) circle (2.9);
\fill[black!20!white] (0.701,1.588) circle (2.1);
\fill[black!40!white] (0,0) circle (1.735);
\draw[black!10!white] (0.701,1.588) circle (2.9);
\draw[black!20!white] (0.701,1.588) circle (2.1);
\fill[blue] (0.701,1.588) circle (2pt);
\coordinate[label=45: {${x_3}$}] (x_3) at (0.701,1.588);
\draw (0.701,1.588) -- (3.212,0.471);
\draw (0.701,1.588) -- (-2.023,1.411);
\draw (0.701,1.588) -- (0.533,3.854);
\draw[gray] (0.713,1.6) -- (3.224,0.483);
\draw[gray] (3.212,0.471) -- (2.083,2.481);
\draw[gray] (0.701,1.588) -- (2.083,2.481);
\fill[blue] (-2.512,1.277) circle (2pt);
\fill[blue] (2.083,2.481) circle (2pt);
\fill[blue] (-1.396,-2.648) circle (2pt);
\fill[blue] (-2.178,3.174) circle (2pt);
\fill[blue] (-2.023,1.411) circle (2pt);\coordinate[label=45: {${x_5}$}] (x_5) at (-2.023,1.411);
\fill[blue] (-3.189,-2.799) circle (2pt);
\fill[blue] (2.254,-4.095) circle (2pt);
\fill[blue] (-1.366,4.344) circle (2pt);
\fill[blue] (-3.339,-1.138) circle (2pt);
\fill[blue] (2.600,0.874) circle (2pt);
\fill[blue] (1.340,-0.892) circle (2pt);\coordinate[label=0: {${x_2}$}] (x_2) at (1.340,-0.892);
\fill[blue] (3.342,3.375) circle (2pt);
\fill[blue] (-2.541,0.548) circle (2pt);
\fill[blue] (-0.100,0.239) circle (2pt);\coordinate[label=0: {${x_1}$}] (x_1) at (-0.100,0.239);
\fill[blue] (-2.958,3.654) circle (2pt);
\fill[blue] (0.634,3.447) circle (2pt);
\fill[blue] (0.533,3.854) circle (2pt);
\fill[blue] (-3.107,0.423) circle (2pt);
\fill[blue] (-1.741,-1.777) circle (2pt);
\fill[blue] (-3.018,0.222) circle (2pt);
\fill[blue] (-3.613,-1.232) circle (2pt);
\fill[blue] (3.212,0.471) circle (2pt);
\fill[blue] (-1.627,3.913) circle (2pt);
\fill[blue] (2.734,-3.825) circle (2pt);
\fill[blue] (3.898,1.760) circle (2pt);
\fill[blue] (3.127,3.340) circle (2pt);
\fill[black] (0,0) circle (1pt);
\coordinate[label=-45: {${o}$}] (o) at (0,0);
\end{tikzpicture}
\caption{Construction of re-summation of vacuum potentials. Indicated numbering of points in configuration $\o$ is according to cyclic ordering; $\L_{x_3}$ given by the complement of the dark-gray area; two balls $\D_{x_3,m}$ and $\D_{x_3,m'}$ are given by middle-gray and light and middle-gray area around $x_3$ including the parts hidden by dark-gray area; $\o_{x,m'}$ is given by the four points which are connected via black lines; example $\eta\in P^\o_{x,m'}$ is given via points in triangle with gray edges, for this $\eta$, $\D(\eta)=\D_{x_3,m'}$.}
\label{PIX_2}  
\end{figure}
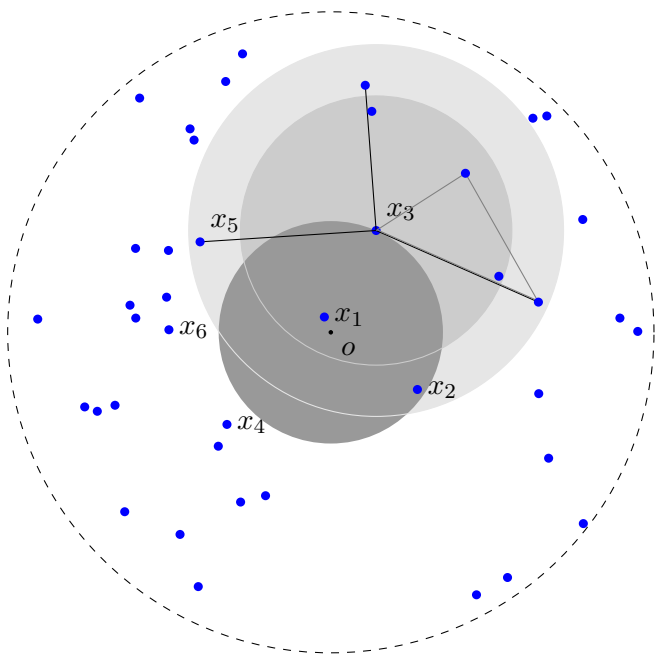

\medskip
What remains to show is that $\Psi$ defines an equivalent Hamiltonian as $\Phi$ and that $\Psi$ is indeed absolutely summable for a good choice of $\D_{x,m}$. W.r.t.~the equivalence, note that 
\begin{equation*}
\begin{split}
\sum_{\eta\Subset\o:\, \eta\cap\L\neq\emptyset}\Psi(\eta,\oo)
&=\sum_{x\in\o_\L}\sum_{m=1}^\infty\Psi(\o_{x,m},\oo)+\sum_{y\in\o:\, y<l(\o_\L)}\sum_{m\in\N:\, \o_{y,m}\cap\L\neq\es}\Psi(\o_{y,m},\oo)\cr
&=\sum_{x\in\o_\L}\sum_{m=1}^\infty\sum_{\eta\in P^\o_{x,m}}\Phi(\eta,\oo)+\sum_{y\in\o:\, y<l(\o_\L)}\sum_{m\in\N:\, \o_{y,m}\cap\L\neq\es}\sum_{\eta\in P^\o_{y,m}}\Phi(\eta,\oo)\cr
&=\sum_{\eta\Subset\o:\, l(\eta)\in\L}\Phi(\eta,\oo)+\sum_{\eta\Subset\o:\, l(\eta)< l(\o_\L), \, \eta\cap\L\neq\es}\Phi(\eta,\oo)+\sum_{\eta\Subset\o:\, \eta\in Q^\o_\L}\Phi(\eta,\oo), 
\end{split}
\end{equation*}
where $Q^\o_\L=\{\eta\Subset\o:\, l(\eta)< l(\o_\L), \,  \eta\cap\L=\es, \text{ there exists }m\in\N \text{ such that } \eta\subset\o_{y,m}\cap\L\neq\es\}$. 
Now the last summand only depends on $\oo_{\L^c}$ and hence $\Psi$ and $\Phi$ are equivalent.

\medskip
W.r.t.~the absolute summability, note that for all $\L\Subset\R^d$ and $\oo\in\O$, 
\begin{equation*}
\begin{split}
\sum_{\eta\Subset\o:\, \eta\cap\L\neq\emptyset}|\Psi(\eta,\oo)|
&=\sum_{x\in\o_\L}\sum_{m=1}^\infty|\Psi(\o_{x,m},\oo)|+\sum_{y\in\o:\, y<l(\o_\L)}\sum_{m\in\N:\, \o_{y,m}\cap\L\neq\es}|\sum_{\eta\in P^\o_{y,m}}\Phi(\eta,\oo)|,
\end{split}
\end{equation*}
where the second summand is finite since all the sums involved are in fact finite. Indeed, the first sum is finite due to the definition of the ordering. The second sum is finite by finiteness of $\L$ and the third sum is finite by the locally finiteness of $\o$ and the assumption that the vacuum Hamiltonian is finite.

In order to prove $\sum_{m=1}^\infty|\Psi(\o_{x,m},\oo)|<\infty$, note that 
by assumption $\Phi$ is uniformly convergent and hence, for every $x\in\R^d$, there exists a co-final sequence $(\D_{x,m})_{m\ge1}$ of balls in $\R^d$ with radius $r_m\in \N$, centered at $x\in\R^d$ such that
\begin{equation*}
\sup_{\oo\in\OO^*}|\sum_{x\in\eta\Subset\o:\, \eta\not\subset{\D_{x,m}}}\Phi({\eta},\oo)|<m^{-2}
\end{equation*}
and in particular, recalling $\L_x=\{y\ge x\}$, we have
\begin{equation*}
\sup_{\oo_{\L_x}\in\OO^*}|\sum_{x\in\eta\Subset\o_{\L_x}:\, \eta\not\subset{A_{x,m}}}\Phi({\eta},\oo)|<m^{-2}.
\end{equation*}
For this choice of $\D_{x,m}$ we have
\begin{equation*}
\begin{split}
\sum_{m=1}^\infty|\Psi(\o_{x,m},\oo)|
&\le |\sum_{\eta\Subset\o:\, l(\eta)=x,\, \eta\subset A_{x,1}}\Phi({\eta},\oo)|+\sum_{m=2}^\infty|\sum_{\eta\Subset\o:\, l(\eta)=x,\, r(\eta)\in A_{x,m}\sm A_{x,m-1} }\Phi({\eta},\oo)|\cr
&\le |\sum_{\eta\Subset\o:\, l(\eta)=x,\, \eta\subset A_{x,1}}\Phi({\eta},\oo)|+2\sum_{m=1}^\infty|\sum_{\eta\Subset\o:\, l(\eta)=x,\, \eta\not\subset A_{x,m}}\Phi({\eta},\oo)|
\end{split}
\end{equation*}
where the first summand consists of only finitely many summands and the second summand is bounded from above by $2\sum_{m\ge1}m^{-2}<\infty$ as required. 

\medskip
Finally note that in order to determine the horizon of $(\eta,\oo)\in\EE$, it suffices to consider the case $\eta=\o_{x,m}$ for some $x\in\R^d$ and $m\ge 1$. But by the definitions, $\Psi(\o_{x,m},\oo)=\Psi(\o_{x,m},\tilde\oo)$ if $\oo_{\D_{x,m}}=\tilde\oo_{\D_{x,m}}$ and hence $\D(\eta,\oo)=\D(\eta)$.
\end{proof}
Let us make a few more comments on the above proof. 
\begin{enumerate}
\item The mapping $\eta\mapsto \D(\eta)$ is measurable on $\mathbb{B}=\{B_{m}(x)| x\in\R^d, m\in\N\}$ with $\s$-algebra $\mathcal{B}(\mathbb{B})=\s(\{B_{n}(x)\in\mathbb{B}| x\in A,n=m\},A\in\mathcal{B}(\R^d),m\in\N)$.
In order to see this, note that the mapping $\eta\mapsto l(\eta)$ is measurable w.r.t.~$\mathcal{B}(\R^d)$ since $\{\eta|\, l(\eta)\in A\}=\{\eta|\, |\eta\cap A|\ge 1\}\cap\{\eta|\, |\eta\cap A'|=0\}$ where $A'=\{x\in\R^d| x<y\text{ for all }y\in A\}$. Further note that we can decompose $\eta\mapsto(l(\eta),\eta)\mapsto \D_{l(\eta)}(\eta)=\D(\eta)$ and 
\begin{equation*}
\begin{split}
\{(l(\eta),\eta)|&\D(\eta)=B_{r_n}(l(\eta))\}=\{(l(\eta),\eta)|\,  |\eta\cap B_{r_n}(l(\eta))^c|=0\}\cap\cr
&\bigcup_{m\in \N}[\{(l(\eta),\eta)|\, |\eta|=m\}\cap \{(l(\eta),\eta)|\, |\eta\cap B_{r_{n-1}}(l(\eta))|\le m-1\}]
\end{split}
\end{equation*}
which shows measurability of the second mapping. Finally, since the mapping $n\mapsto r_n$, $\N\to\N$ is trivially measurable, the result follows. 

\item Instead of balls, the co-final sequence $\D_{x,m}$ can also consist of measurable sets. Also in this case measurability of $\eta\mapsto \D(\eta)$ follows by measurability of $\Phi$.

\item $\Psi$ in general does not have $r$-uniform finite horizons. Indeed, $\D(\o_{x,m})$ has only points in $x$ and in the annulus, but it is composed from vacuum potentials with points in the whole $\D_{x,m}$ and hence it is not sufficient to know the $r$-vicinity of $\o_{x,m}$. 

\item Let us also note, that the proof of Theorem~\ref{Representation_Absolute} is not easily adaptable to give absolutely-summable potentials with finite horizon property in the absence of uniformly convergent vacuum potentials. The reason for this is that the co-final sequence $\D_{x,m}$ (which is designed to give a sufficiently quick exhaustion of $\R^d$ such that summability follows) would depend on $\oo$. But then the finite horizon property can not be guaranteed any more.

\item  As can be seen from the proof, the hypergraph structure for $\Psi$ is given by $\EE^*=\{(\eta,\oo)\in\bar\EE^*:\, \eta=\o_{l(\eta),m}\text{ for some }m\in\N\}$.

\item The absolutely-summable potential $\Psi$ is not uniformly convergent in general. To see this, note that the speed of convergence to $\R^d$ of the co-final sequence $\D_{x,m}$ in general depends on $x$. Hence, due to possible accumulations of points in $\L$, the energy contribution of long-range potentials is not guaranteed to be uniformly small.

\item Finally, note that in general there is no absolute summability with a uniform bound, see for example the Potts gas. 
\end{enumerate}

The ordering of points in $\R^d$ used in the above proof for absolute summability is such that the resulting potential is rotation invariant on the set of configurations where no pair of points can lie on the same sphere. This set has of course probability one under the Poisson point process and is thus negligible and can be excluded from $\OO^*$. The property of rotational invariance is meaningful only in the continuum and not for lattice systems. However, translation invariance of the potential is a meaningful and an often desirable property both for lattice systems and systems of point particles, and the potential constructed above is not translation invariant due to the choice of the ordering. On the other hand, this ordering allows to have only finitely many potentials which come from interactions with points smaller then any given point in the non-translation-invariant ordering. This is convenient for the proof of summability and allowed us to only assume vacuum uniformly log-quasilocality of the Poisson modification. In order to keep the translation invariance provided by the vacuum potential, the ordering must be such that translation invariance is retained, i.e., a group ordering on $\R^d$. But then, for any point in a configuration $\oo$, there are potentially infinitely many smaller points in a configuration $\oo$. To guarantee the absolute summability in this situation, we use the summable-vacuum-uniform quasilocality which ensures sufficiently fast convergence of the vacuum logarithmic Poisson modification on point configurations with a bounded density of points. 
\begin{proof}[Proof of Theorem~\ref{Representation_Absolute_2}]
We use the regrouping approach of \cite[Theorem 3]{Ko74} adapted to the continuous setting. 
Let now denote 
$\L_m=\{x\in\R^d:\, \Vert x\Vert\leq m\}$ with 
$m\in\N$ and $\Vert\cdot\Vert$ an arbitrary norm on $\R^d$. We write $\L_{x,m}=\L_m+x$ for the shift of the box by $x\in\R^d$. Further, let $\le$ be any group ordering on $\R^d$, for example the lexicographical ordering, and define $A^\bullet _{x,m}=\L_{x,m}\cap\{y\bullet x\}$ with $A^\bullet_{x,0}=\es$ where $\bullet\in\{\ge, >,\le,< \}$. 
As before, for $\eta\Subset\R^d$ we will write $l(\eta)$ and $r(\eta)$ to denote the left and right end points of $\eta$ in the given ordering. Again, we define the grading
$$P_{x,m}=\{\eta\Subset\R^d:\, l(\eta)=x \text{ and } r(\eta)\in(A^\ge_{x,m}\sm A^\ge_{x,m-1})\},$$
the set of finite subsets of $\R^d$ with left end point equal to $x$ and right end point in the $m$-annulus of the sequence $A^\ge_{x,m}$. 
Note that in particular, $\bigcup_{x,m}P_{x,m}=\{\eta:\, \eta\Subset\R^d\}$ is a disjoint partition of the set of finite subsets of $\R^d$. 

As before, we perform the regrouping w.r.t.~the unique vacuum potential $\Phi$ as 
\begin{equation*}
\Psi({\o_{x,m}},\oo)=\sum_{\eta\in P^\o_{x,m}}\Phi(\eta,\oo),
\end{equation*}
where 
$P^\o_{x,m}=P_{x,m}\cap\{\eta\Subset\o\}$ 
and $\o_{x,m}=\o\cap (\{x\}\cup (A^\ge_{x,m}\sm A^\ge_{x,m-1}))$ and $\Psi(\eta,\oo)=0$ otherwise, see Figure~\ref{PIX_3} for an illustration.
Again, note that we have $\Psi({\o_{x,m}},\oo)=\Psi({\o_{x,m}},\tilde\oo)$ if $\o_{A^\ge_{x,m}}=\tilde\o_{A^\ge_{x,m}}$ and thus $\Psi$ has the uniform finite-horizon property.
Moreover, note that $\Psi$ is translation invariant, by the translation invariance of $\Phi$ and the fact that $A^\ge_{x,m}$ is translation invariant.

\begin{figure}[!htpb]
\centering
\begin{tikzpicture}[scale=0.85]
\draw[black,dashed] (0,0) circle (5);
\coordinate[label=-180: {${x_{i}}$}] (x_1) at (-0.100,0.239);
\begin{scope}
\clip(-0.100,-5) rectangle (5,5);
\draw [black!20!white] (-0.100,0.239) circle (0.85);
\draw [black!20!white] (-0.100,0.239) circle (1.7);
\draw [black!20!white] (-0.100,0.239) circle (2.55);
\draw [black!20!white] (-0.100,0.239) circle (3.4);
\draw [black!20!white] (-0.100,0.239-3.4) rectangle (-0.100,0.239+3.4);
\end{scope}
\coordinate[label=0: {${x_{i+1}}$}] (x_2) at (0.533,3.854);
\coordinate[label=0: {${x_{i+2}}$}] (x_4) at (0.634,3.447);
\coordinate[label=0: {${x_{i+3}}$}] (x_5) at  (1.340,-0.892);
\coordinate[label=0: {${x_{i+4}}$}] (x_6) at (2.083,2.481);
\coordinate[label=90: {${x_{i+5}}$}] (x_7) at (2.254,-4.095);
\fill[blue] (-0.100,0.239) circle (2pt);
\fill[blue] (-1.010,-2.549) circle (2pt);
\fill[blue] (-1.366,4.344) circle (2pt);
\fill[blue] (-1.396,-2.648) circle (2pt);
\fill[blue] (-1.607,-1.437) circle (2pt);
\fill[blue] (-1.627,3.913) circle (2pt);
\fill[blue] (-1.741,-1.777) circle (2pt);
\fill[blue] (-2.023,1.411) circle (2pt);
\fill[blue] (-2.051,-3.967) circle (2pt);
\fill[blue] (-2.116,2.999) circle (2pt);
\fill[blue] (-2.178,3.174) circle (2pt);
\fill[blue] (-2.333,-3.153) circle (2pt);
\fill[blue] (-2.505,0.040) circle (2pt);
\fill[blue] (-2.512,1.277) circle (2pt);
\fill[blue] (-2.541,0.548) circle (2pt);
\fill[blue] (-2.958,3.654) circle (2pt);
\fill[blue] (-3.018,0.222) circle (2pt);
\fill[blue] (-3.021,1.309) circle (2pt);
\fill[blue] (-3.107,0.423) circle (2pt);
\fill[blue] (-3.189,-2.799) circle (2pt);
\fill[blue] (-3.339,-1.138) circle (2pt);
\fill[blue] (-3.613,-1.232) circle (2pt);
\fill[blue] (-3.808,-1.164) circle (2pt);
\fill[blue] (-4.535,0.206) circle (2pt);
\fill[blue] (0.533,3.854) circle (2pt);
\fill[blue] (0.634,3.447) circle (2pt);
\fill[blue] (1.340,-0.892) circle (2pt);
\fill[blue] (2.083,2.481) circle (2pt);
\draw[thick] (2.083,2.481) -- (-0.100,0.239);
\draw[gray] (2.083,2.481) -- (-0.100,0.239);
\draw[gray] (1.340,-0.892) -- (-0.100,0.239);
\draw[gray] (1.340,-0.892) -- (2.083,2.481);
\draw[thick] (-0.100,0.239) -- (2.714,1.797);
\draw[thick] (-0.100,0.239) -- (2.600,0.874);
\fill[blue] (2.254,-4.095) circle (2pt);
\fill[blue] (2.600,0.874) circle (2pt);
\fill[blue] (2.714,1.797) circle (2pt);
\fill[blue] (3.127,3.340) circle (2pt);
\fill[blue] (3.212,0.471) circle (2pt);
\fill[blue] (3.217,-0.957) circle (2pt);
\fill[blue] (3.342,3.375) circle (2pt);
\fill[blue] (3.370,-1.964) circle (2pt);
\fill[blue] (3.898,1.760) circle (2pt);
\fill[blue] (3.907,-2.984) circle (2pt);
\fill[blue] (4.473,0.222) circle (2pt);
\fill[blue] (4.749,0.014) circle (2pt);
\end{tikzpicture}
\caption{Construction of re-summation of vacuum potentials w.r.t.~the lexicographical ordering. Indicated numbering of points in configuration $\o$ started from the central point with index $i\in \Z$; light-gray lines indicate the sets $A^\ge_{x_i,1},\dots, A^\ge_{x_i,4}$; the solid lines are examples of two-point sets in $P_{x_i,4}$ and the triangle $x_i,x_{i+3},x_{i+4}$ is an example of a three-point set in $P_{x_i,4}$.}
\label{PIX_3}  
\end{figure}
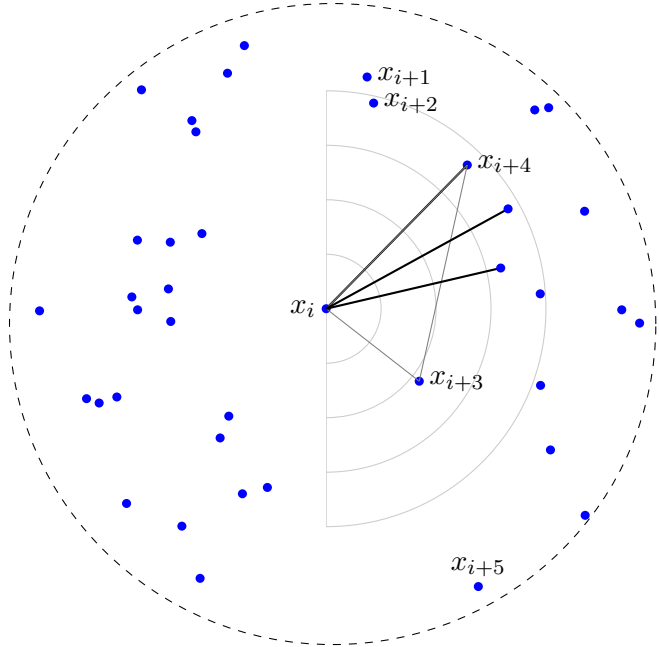

\medskip
What remains to show is that $\Psi$ defines an equivalent Hamiltonian as $\Phi$ and that $\Psi$ is indeed absolutely summable. The equivalence can be shown precisely as in the proof of Theorem~\ref{Representation_Absolute}. 
For the absolute summability, note that for all $\L\Subset\R^d$ and $\oo\in\OO^*$, 
\begin{equation*}
\begin{split}
\sum_{\eta\Subset\o:\, \eta\cap\L\neq\emptyset}|\Psi(\eta,\oo)|
&=\sum_{x\in\o_\L}\sum_{m=1}^\infty|\Psi(\o_{x,m},\oo)|+\sum_{y\in\o:\, y<l(\o_\L)}\sum_{m\in\N:\, \o_{y,m}\cap\, \o_\L\neq\es}|\Psi(\o_{y,m},\oo)|\cr
&\le \sum_{x\in\o_\L}\Big(\sum_{m=1}^\infty|\Psi(\o_{x,m},\oo)|+\sum_{y\in\o:\, y<l(\o_\L)}\sum_{m\in\N:\, \o_{y,m}\ni x}|\Psi(\o_{y,m},\oo)|\Big)\cr
&= \sum_{x\in\o_\L}\Big(\sum_{m=1}^\infty|\Psi(\o_{x,m},\oo)|+\sum_{y\in\o:\, y<l(\o_\L)}|\Psi(\o_{y,m_{y,x}},\oo)|\Big)\cr
&= \sum_{x\in\o_\L}\sum_{m=1}^\infty\Big(|\Psi(\o_{x,m},\oo)|+\sum_{y\in\o:\, m_{y,x}=m}|\Psi(\o_{y,m},\oo)|\Big)\cr
\end{split}
\end{equation*}
where $m_{y,x}$ denotes the integer $m$ such that $x\in A^\ge_{y,m}\sm A^\ge_{y,m-1}$. The integer $m_{y,x}$, by translation invariance, only depends on the relative position of $x$ and $y$. Hence, $m_{y,x}$ is equivalently given also by the $m$ such that $y\in A^<_{x,m}\sm A^<_{x,m-1}$. Thus, 
\begin{equation*}
\begin{split}
\sum_{\eta\Subset\o:\, \eta\cap\L\neq\emptyset}|\Psi(\eta,\oo)|
&\le \sum_{x\in\o_\L}\sum_{m=1}^\infty \Big(|\Psi(\o_{x,m},\oo)|+\sum_{y\in \o_{A^<_{x,m}\sm A^<_{x,m-1}}}|\Psi(\o_{y,m},\oo)|\Big).
\end{split}
\end{equation*}
Further note that we can rewrite $\Psi$ as
\begin{equation*}
\begin{split}
\Psi(\o_{y,m},\oo)
&=\sum_{\eta\subset\o:\, l(\eta)=y,\, r(\eta)\in A^\ge _{y,m}\sm A^\ge_{y,m-1} }\Phi({\eta},\oo)\cr
&=\sum_{\eta\subset\o:\, l(\eta)=y,\, r(\eta)\in A^\ge _{y,m}}\Phi({\eta},\oo)-\sum_{\eta\subset\o:\, l(\eta)=y,\, r(\eta)\in A^\ge_{y,m-1} }\Phi(\eta,\oo)\cr
&=\sum_{\eta\subset\o:\, l(\eta)=y,\, \eta\subset A^\ge _{y,m}}\Phi({\eta},\oo)-\sum_{\eta\subset\o:\, l(\eta)=y,\, \eta\subset A^\ge_{y,m-1} }\Phi({\eta},\oo)\cr
&=\sum_{\eta\subset\o_{A^\ge_{y,m}}:\, l(\eta)=y}\Phi({\eta},\oo)-\sum_{\eta\subset\o_{A^\ge_{y,m-1}}:\, l(\eta)=y}\Phi({\eta},\oo)
\end{split}
\end{equation*}
where 
\begin{equation*}
\begin{split}
\sum_{\eta\subset\o_{A^\ge_{y,m}}:\, l(\eta)=y}\Phi({\eta},\oo)=\sum_{\eta\subset\o_{A^\ge_{y,m}}}\Phi({\eta},\oo)-\sum_{\eta\subset\o_{A^>_{y,m}}}\Phi({\eta},\oo).
\end{split}
\end{equation*}
Now, for the vacuum potential $\Phi$, 
see~\eqref{VacPot}, we have 
\begin{equation*}
\begin{split}
\sum_{\eta\subset\o_{A^\ge_{y,m}}}\Phi({\eta},\oo)=-\log\frac{\r_{A^\ge_{y,m}}(\oo_{A^\ge_{y,m}})}{\r_{A^\ge_{y,m}}(\es_{A^\ge_{y,m}})}
\end{split}
\end{equation*}
and hence by the pre-modification property of $\r$, we have
\begin{equation*}
\begin{split}
\sum_{\eta\subset\o_{A^\ge_{y,m}}:\, l(\eta)=y}\Phi({\eta},\oo)=-\log\frac{\r_y(\oo_{A^>_{y,m}})}{\r_y(\oo_{A^\ge_{y,m}})}.
\end{split}
\end{equation*}
Thus, by translation invariance and the assumption of summable-vacuum-uniform log-quasilocality, for $m\ge 1$, 
\begin{equation*}
\begin{split}
|\Psi(\o_{y,m},\oo)|=|\log\frac{\r_y(\oo_{A^>_{y,m}})}{\r_y(\oo_{A^\ge_{y,m}})}-\log\frac{\r_y(\oo_{A^>_{y,m-1}})}{\r_y(\oo_{A^\ge_{y,m-1}})}|\le 2\k(m-1)
\end{split}
\end{equation*}
and hence
\begin{equation*}
\begin{split}
\sum_{\eta\Subset\o:\, \eta\cap\L\neq\emptyset}|\Psi(\eta,\oo)|
&\le 2\sum_{x\in\o_\L}\sum_{m=1}^\infty\big(\k(m-1)+|\o_{A^<_{x,m}\sm A^<_{x,m-1}}| \k(m-1)\big).
\end{split}
\end{equation*}
The first summand on the right-hand side is finite, since $|\o_\L|<\infty$ and $\sum_{m=1}^\infty\k(m-1)<\infty$. For the second summand, recall that, by assumption we derive a representation only for $\oo$ satisfying the density constraint $\sup_{m\in\N}m^{-d}|\o_{\L_m}|<C(\o)<\infty$. This allows us to further bound
\begin{equation*}
\begin{split}
\sum_{x\in\o_\L}\sum_{m=1}^\infty|\o_{A^<_{x,m}\sm A^<_{x,m-1}}|\k(m-1)\le |\o_\L|C(\o)\sum_{m=1}^\infty m^d\k(m-1)<\infty,
\end{split}
\end{equation*}
which completes the proof.
\end{proof}
Let us remark that different choices for the norm $\Vert\cdot\Vert$, used in the definition of the annuli-defining balls $\L_n$ in the above proof, yield different additional symmetries. For example, for the euclidean norm, the construction implies the preservation of symmetries
of the pre-modification $\rho$ to the associated translation-invariant potential $\Psi$ w.r.t.~ 
all euclidean symmetries which keep the positive $e_1$-axis invariant. In particular, all rotations around the $e_1$-axis can be preserved. Note also that the $e_1$-direction can be replaced by an arbitrary other direction in euclidean space.

\begin{proof}[Proof of Lemma~\ref{VacPotNonZero}]
Assume that $\eta$ consists of two clusters $\eta_1,\eta_2\neq\es$, then
we have 
\begin{equation*}
\begin{split}
-\Phi({\eta},\oo)
&=\sum_{\xi_1\subset \eta_1}(-1)^{|\eta_1\sm\xi_1|}\sum_{\xi_2\subset \eta_2} (-1)^{|\eta_2\sm\xi_2|} \sum_{C\in\CC(\xi_1\cup\xi_2)}\Psi(C,\oo)\cr
&=\sum_{\xi_1\subset \eta_1}(-1)^{|\eta_1\sm\xi_1|}\sum_{\xi_2\subset \eta_2} (-1)^{|\eta_2\sm\xi_2|}(\Psi(C,\oo_{\xi_1})+\Psi(C,\oo_{\xi_2}))\cr
&=\sum_{\xi_1\subset \eta_1}(-1)^{|\eta_1\sm\xi_1|}\sum_{\xi_2\subset \eta_2} (-1)^{|\eta_2\sm\xi_2|}\Psi(C,\oo_{\xi_2})=0.
\end{split}
\end{equation*}
\end{proof}

\begin{proof}[Proof of Lemma~\ref{AsymptoticPot}]
Let us start by estimating $\phi$ using additional cross terms. 
\begin{equation*}
\begin{split}
|\phi(n)|
&=|\sum_{j=1,3,\dots}^{\infty}\big(\frac{1}{j}(1-a^{j})^{n}- \frac{1}{j+1}(1-a^{j+1})^{n}\big)|\cr
&\le|\sum_{j=1,3,\dots}^{\infty}\frac{1}{j}\big((1-a^{j})^{n}-(1-a^{j+1})^{n})\big)|+| \sum_{j=1,3,\dots}^{\infty}  (\frac{1}{j}- \frac{1}{j+1})(1-a^{j+1})^{n}|\cr
&\le\sum_{j\ge 1}^{\infty}\frac{1}{j}\big((1-a^{j+1})^{n}-(1-a^{j})^{n})\big)+ \sum_{j\ge 1}^{\infty}\frac{1}{j(j+1)}(1-a^{j+1})^{n}\cr
&\le (1-a)^n+2\sum_{j\ge 1}^{\infty}\frac{1}{j(j+1)}(1-a^{j+1})^{n}
\end{split}
\end{equation*}
The first term decays exponentially. In order to determine the asymptotic behaviour of the second term, let us split the sum into terms $j\geq J$ and $j<J$, 
with $J=J(n)$ tending to infinity with $n$, in a way chosen below. 
We obtain the upper bound 
\begin{equation*}
\begin{split}
\sum_{j\geq J} \frac{1}{j(j+1)}(1-a^{j+1})^{n}+ \sum_{j<J}\frac{1}{j(j+1)}(1-a^{j+1})^{n}&\leq \frac{1}{J}+ (1-a^{J})^{n}
\end{split}
\end{equation*}
where we used twice that $\sum_{j\geq J} (j(j+1))^{-1}=J^{-1}$. 
As a final step, we optimize over $J$ given $n$ in such a way that the expression
$$(1-a^{J})^{n}
=\exp\big(-n(a^J+o(a^J))\big)$$ 
tends to zero as $n$ tends to infinity. In order to achieve this, take $na^J=n^\e$ for arbitrary $\e>0$. Then $J(n)=((\e-1)/\log a)\log n$ which gives the desired speed of convergence with $C>2\log(1/a)$. 
\end{proof}

Convenient choices are for instants the $\Vert\cdot\Vert_\infty$-norm, or, on the other hand, the euclidean norm, which gives cylindric symmetry for the re-summed potentials, to be constructed as follows.

\begin{proof}[Proof of Proposition~\ref{AbsSumWRM}]
It suffices to consider 
\begin{equation*}
\begin{split}
&\sup_{\oo\in\OO}|\log h_\L(\oo)-\log h_\L(\oo_\D)|\cr&=\sup_{\oo\in\OO}|\sum_{C\in\CC^{\rm f}_\L(\o)}\log(1+a^{|\oo_C|^+}b^{|\oo_C|^-})-\sum_{C\in\CC^{\rm f}_\L(\o_\D)}\log(1+a^{|\oo_C|^+}b^{|\oo_C|^-})|
\end{split}
\end{equation*}
with $\D\supset\L$. Let $d(\L,\D^c)$ denotes the set distance between $\D^c$ and $\L$. Note that, in order for a cluster in $\CC^{\rm f}_\L(\o)$ not to be also contained in $\CC^{\rm f}_\L(\o_\D)$ it must at least have $d(\L,\D^c)/(2r)$ many points since otherwise is would be contained in $\D$. Conversely, every cluster in $\CC^{\rm f}_\L(\o_\D)$ is part of a cluster in $\CC^{\rm f}_\L(\o)\cup\CC^{\infty}_\L(\o)$. If the cluster would be contained in $\D$, then both contributions cancel. Hence, non canceling clusters in $\CC^{\rm f}_\L(\o_\D)$ must also have at least $d(\L,\D^c)/(2r)$ many points. Moreover, there can only be $K=K(\L, r)$ different clusters attached to $\L$. Hence with $c=a\vee b$ we have 
\begin{equation*}
\begin{split}
&\sup_{\oo\in\OO}|\log h_\L(\oo)-\log h_\L(\oo_\D)|\le 2K\log(1+c^{d(\L,\D^c)/(2r)})
\end{split}
\end{equation*}
which tends to zero as $\D$ tends to $\R^d$ since in the Gibbsian regime $t>t_G$ we have $a,b<1$.
\end{proof}

\begin{proof}[Proof of Proposition~\ref{AbsSumTransWRM}]
As can be seen from the last estimate in the proof of Proposition~\ref{AbsSumWRM}, there exists $K>0$ and $0<s<1$ such that
\begin{equation*}
\begin{split}
&\sup_{\oo\in\OO}|\log h_0(\oo)-\log h_0(\oo_{\L_n})|\le 2K\log(1+s^n),
\end{split}
\end{equation*}
but $\sum_{n\ge 1}n^d\log(1+s^n)<\infty$, which can be seen for example via the integral test for convergence.
\end{proof}

\bibliography{Jahnel}
\bibliographystyle{alpha}

\end{document}